\documentclass[12pt]{amsart}

\usepackage{amsfonts, amssymb, amsmath, amsthm, amsrefs, enumerate}

\usepackage{xcolor}
\usepackage{hyperref}
\hypersetup{
  colorlinks=true,
  linkcolor=blue,
  citecolor=red
}

\theoremstyle{definition}
\newtheorem{definition}{Definition}

\theoremstyle{plain}
\newtheorem{theorem}[definition]{Theorem}
\newtheorem{proposition}[definition]{Proposition}

\newtheorem*{claim}{Claim}

\newcommand{\suchthat}{:}
\newcommand{\from}{\colon}

\renewcommand{\epsilon}{\varepsilon}

\newcommand{\N}{\mathbb{N}}
\newcommand{\symdiff}{\mathrel{\triangle}}
\newcommand{\R}{\mathbb{R}}

\DeclareMathOperator{\cost}{cost}

\begin{document}
  \title[Unfriendly colorings]{Unfriendly colorings of graphs with finite average degree}

  \author{Clinton T.~Conley}
  \address[C.T.~Conley]{Carnegie Mellon University.}
  \author{Omer Tamuz}
  \address[O.~Tamuz]{California Institute of Technology.}
\thanks{Clinton T.~Conley was supported by NSF grant DMS-1500906. Omer Tamuz was supported by a grant from the Simons Foundation (\#419427).}

  \maketitle

\begin{abstract}
    In an unfriendly coloring of a graph the color of every node mismatches that of the majority of its neighbors. We show that every probability measure preserving Borel graph with finite average degree admits a Borel unfriendly coloring almost everywhere. We also show that every bounded degree Borel graph of subexponential growth admits a Borel unfriendly coloring.
\end{abstract}

\section{Introduction}

Suppose that $G$ is a locally finite graph on the vertex set $X$.  We say that $c \from X \to 2$ is an \emph{unfriendly coloring} of $G$ if for all $x \in X$ at least half of $x$'s neighbors receive a different color than $x$ does.  More formally, letting $G_x$ denote the set of $G$-neighbors of $x$, such a function $c$ is an unfriendly coloring if $|\{y \in G_x \suchthat c(x) \neq c(y)\}| \geq |\{y \in G_x \suchthat c(x) = c(y)\}|$. By a compactness argument unfriendly colorings exist for all locally finite graphs (see, e.g.,~\cite{aharoni1990unfriendly}). There exist graphs with uncountable vertex sets that have no unfriendly colorings \cite{shelah1990graphs};  it is not known if this is possible for graphs with countably many vertices.

A large and growing literature considers measure-theoretical analogues of classical combinatorial questions (see, e.g., a survey by Kechris and Marks \cite{kechris2015descriptive}). 
Following~\cite{conley2014measure}, we consider a measure-theoretical analogue of the question of unfriendly colorings. Suppose that $G$ is a locally finite Borel graph on the standard Borel space $X$, and that $\mu$ is a Borel probability measure on $X$.  We say that $G$ is \emph{$\mu$-preserving} if there are countably many $\mu$-preserving Borel involutions whose graphs cover the edges of $G$.  Equivalently, $G$ is $\mu$-preserving if its connectedness relation $E_G$ is a $\mu$-preserving equivalence relation.

An important example of such graphs comes from probability measure preserving actions of finitely generated groups. Indeed let  a group, generated by the finite symmetric set $S$, act by measure preserving transformations on a standard Borel probability space $(X,\mu)$. Then the associated graph $G=(X,E)$ whose edges are
$$
 E = \{(x,y) \,:\, y=s x\text{ for some } s \in S\}
$$ 
is a $\mu$-preserving graph.

In \cite{conley2014measure} it is shown that every free probability measure preserving action of a finitely generated group is weakly equivalent to another such action whose associated graph admits an unfriendly coloring. Note that such graphs are regular: (almost) every node has degree $|G_x|=|S|$. Recall that the ($\mu$-)\emph{cost} of a $\mu$-preserving locally finite Borel graph $G$ is simply half its average degree: $\cost(G) = \frac{1}{2} \int_X |G_x| \ d\mu$.  Equivalently, using the Lusin-Novikov uniformization theorem (see, e.g., \cite[Lemma 18.12]{kechris2012classical}) one may circumvent this factor of $\frac{1}{2}$ by instead computing $\int_X |\vec{G}_x|d\mu$, where $\vec{G}$ is an arbitrary measurable orientation of $G$.

Our first result shows that every measure  preserving graph with finite cost admits an (almost everywhere) unfriendly coloring.

\begin{theorem}\label{thm:invariant}
  Suppose that $(X,\mu)$ is a standard probability space and that $G$ is a $\mu$-preserving locally finite Borel graph on $X$ with finite cost.  Then there is a $\mu$-conull $G$-invariant Borel set $A$ such that $G\restriction A$ admits a Borel unfriendly coloring.
\end{theorem}
We next explore how the invariance assumption can be weakened. Recall that a Borel probability measure is $G$-\emph{quasi-invariant} if the $G$-saturation of every $\mu$-null set remains $\mu$-null.  Such measures admit a \emph{Radon-Nikodym} cocycle $\rho \from G \to \R^+$ so that whenever $A \subseteq X$ is Borel and $f \from A \to X$ a Borel partial injection whose graph is contained in $G$, then $\mu(f[A]) = \int_A \rho(x,f(x))\ d\mu$.

\begin{theorem}\label{thm:quasiinvariant}
  Suppose that $(X,\mu)$ is a standard probability space, that $G$ is a  Borel graph on $X$ with bounded degree $d$, and that $\mu$ is $G$-quasi-invariant, with corresponding Radon-Nikodym cocycle $\rho$. Suppose also that for all $(x,y) \in G$, $1-\frac{1}{d} \leq \rho(x,y) \leq 1+\frac{1}{d}$. Then there is a $\mu$-conull $G$-invariant Borel set $A$ such that $G\restriction A$ admits a Borel unfriendly coloring.
\end{theorem}
The proofs of Theorems~\ref{thm:invariant} and~\ref{thm:quasiinvariant} build on a potential function technique used in~\cite{tamuz2015majority} (see also \cite{benjamini2016convergence}) to study majority dynamics on infinite graphs; in the context of finite graphs, these techniques go back to Goles and Olivos \cite{goles1980periodic}. Indeed, we show that in our settings (anti)-majority dynamics converge to an unfriendly coloring. The combinatorial nature of this technique allows us to extend our results to the Borel setting. 
\begin{theorem}
\label{thm:subexp}
  Suppose that $G$ is a bounded-degree Borel graph of subexponential growth.  Then $G$ admits a Borel unfriendly coloring.
\end{theorem}

A natural question remains open: is there a locally finite Borel graph that does not admit a Borel unfriendly coloring? To the best of our knowledge this is not known, even with regards to the restricted class of bounded degree graphs. In contrast, Theorem~\ref{thm:invariant} shows that for this class unfriendly colorings exist in the measure preserving case. Still, we do not know if the finite cost assumption in Theorem~\ref{thm:invariant} is necessary, or whether every locally finite measure preserving graph admits an almost everywhere unfriendly coloring.

\section{Proofs}
\begin{proof}[Proof of Theorem~\ref{thm:invariant}]
  By Kechris-Solecki-Todorcevic \cite[Proposition 4.5]{kechris1999borel}, there exists a \emph{repetitive sequence of independent sets}: a sequence $(X_n)_{n \in \N}$ of $G$-independent Borel sets so that each $x \in X$ is in infinitely many $X_n$.  We will recursively build for each $n \in \N$ a Borel function $c_n \from X \to 2$ which converge $\mu$-almost everywhere to an unfriendly coloring of $G$.

  The choice of $c_0$ is arbitrary, but we may as well declare it to be the constant $0$ function.

  Suppose now that $c_n$ has been defined.  We build $c_{n+1}$ by ``flipping'' the color of vertices in $X_n$ with too many neighbors of the same color, and leaving everything else unchanged.  More precisely, $c_{n+1}(x) = 1-c_n(x)$ if $x \in X_n$ and $|\{y \in G_x \suchthat c_n(x) \neq c_n(y)\}| < |\{y \in G_x \suchthat c_n(x) = c_n(y)\}|$; otherwise, $c_{n+1}(x) = c_n(x)$.

  To show that this sequence $c_n$ converges $\mu$-a.e.~to an unfriendly coloring, we introduce some auxiliary graphs.  Let $G_n$ be the subgraph of $G$ containing exactly those edges between vertices of the same $c_n$-color, so $x \mathrel{G_n} y$ iff $x \mathrel{G} y$ and $c_n(x) = c_n(y)$.  Certainly for all $n \in \N$, $\cost(G_n) \leq \cost(G)$.

  For $n \in \N$, let $B_n = \{x \in X \suchthat c_n(x) \neq c_{n+1}(x)\}$.

  \begin{claim}
    $\cost(G_n) - \cost(G_{n+1}) \geq \mu(B_n)$.
  \end{claim}

  \begin{proof}[Proof of the claim]
    Recall that, by the definition of $c_{n+1}$, $x \in B_n$ iff $x \in X_n$ and $|\{y \in G_x \suchthat c_n(x) \neq c_n(y)\}| < |\{y \in G_x \suchthat c_n(x) = c_n(y)\}|$.  In particular, $B_n \subseteq X_n$ and hence is $G$-independent.  Thus $G_{n+1} = G_n \symdiff \{(x,y) \suchthat x \mathrel G y$ and $\{x,y\} \cap B_n \neq \emptyset\}$.  But for each $x \in B_n$, the above characterization of membership in $B_n$ ensures that its $G_{n+1}$-degree is strictly smaller than its $G_n$-degree.  The claim follows.
  \end{proof}

  In particular, since the sum telescopes we see $\sum_{n \in \N} \mu(B_n) \leq \cost(G) < \infty$.  Hence the set $C = \{x \in X \suchthat x \in B_n \mbox{ for infinitely many n}\}$ is $\mu$-null by the Borel-Cantelli lemma.  Let $A = X \setminus [C]_G$, so $A$ is $\mu$-conull.

  \begin{claim}
    $c$ is an unfriendly coloring of $G \restriction A$.
  \end{claim}

  \begin{proof}[Proof of the claim]
    Fix $x \in A$ and fix $k\in \N$ sufficiently large so that $c_n$ has stabilized for $x$ and all its (finitely many) neighbors beyond $k$.  Fix $n > k$ so that $x \in X_n$.  Since $c_n(x) = c_{n+1}(x)$, the definition of $c_{n+1}$ implies that $|\{y \in G_x \suchthat c_n(x) \neq c_n(y)\}| \geq |\{y \in G_x \suchthat c_n(x) = c_n(y)\}|$.  But $c_n = c$ on $G_x \cup \{x\}$, and hence $c$ is unfriendly as desired.
  \end{proof}

  This completes the proof of the theorem.
\end{proof}

We next analyze the extent to which the measure-theoretic hypotheses may be weakened in this argument.  Note that the sequence $c_n$ of colorings is defined without using the measure at all (in fact it is determined by the graph $G$ and the sequence $(X_n)$ of independent sets); the measure only appears in the argument that sequence converges to a limit coloring.  And even in this convergence argument, invariance only shows up in the critical estimate $\cost(G_n) - \cost(G_{n+1}) \geq \mu(B_n)$.

\begin{definition}
  Suppose that $G$ is a locally finite Borel graph on standard Borel $X$, that $(X_n)_{n \in \N}$ is a sequence of $G$-independent Borel sets so that each $x\in X$ is in infinitely many $X_n$.  We define the \emph{flip sequence} $(c_n)_{n \in \N}$ of Borel functions from $X$ to $2$ as follows:
  \begin{itemize}
    \item
    $c_0$ is the constant $0$ function,

    \item
    $c_{n+1}(x) = 1-c_n(x)$ if $x \in X_n$ and $|\{y \in G_x \suchthat c_n(x) \neq c_n(y)\}| < |\{y \in G_x \suchthat c_n(x) = c_n(y)\}|$; otherwise, $c_{n+1}(x) = c_n(x)$.
  \end{itemize}
\end{definition}

\begin{definition}
  Given a locally finite Borel graph $G$ on $X$ and a sequence $(X_n)_{n \in \N}$ of repetitive independent sets as above, we say that a Borel measure $\mu$ on $X$ is \emph{compatible} with $G$ and $(X_n)$ if the corresponding flip sequence $c_n$ converges on a $\mu$-conull set.
\end{definition}

The proof of Theorem \ref{thm:invariant} shows that whenever $\mu$ is a $G$-invariant Borel probability measure with respect to which the average degree of $G$ is finite, then $\mu$ is compatible with every sequence of independent sets.  We seek to weaken the invariance assumption when $G$ has bounded degree.

\begin{proposition}\label{prop:quasiinv}
  Suppose that $G$ is a Borel graph on $X$ with bounded degree $d$, and that $\mu$ is a $G$-quasi-invariant Borel probability measure with corrsponding Radon-Nikodym cocycle $\rho$.  Suppose further that for all $(x,y) \in G$, $1-\frac{1}{d} \leq \rho(x,y) \leq 1+\frac{1}{d}$.  Then $\mu$ is compatible with every repetitive sequence of independent sets.
\end{proposition}
Theorem~\ref{thm:quasiinvariant} is an immediate consequence of this proposition.
\begin{proof}[Proof of Proposition~\ref{prop:quasiinv}]
  Put $\epsilon = \frac{1}{d}$.  Define a measure $M$ on $G$ by putting for all Borel $H \subseteq G$,
  $$
  M(H) = \int_X |H_x|\ d\mu
  $$
  This new measure $M$ will replace the occurrences of cost in the proof of Theorem \ref{thm:invariant}.

  Consider the flip sequence $c_n$, and define corresponding graphs $G_n \subseteq G$ by $x \mathrel{G_n} y$ iff $x \mathrel{G} y$ and $c_n(x) = c_n(y)$.  As before, let $B_n$ denote those $x \in X_n$ for which $c_{n+1}(x) \neq c_n(x)$.  Note that the ``double counting'' that occurred in the proof of Theorem \ref{thm:invariant} may no longer be true double counting, but the bound on $\rho$ ensures that each edge is counted at most $(2+\epsilon)$ times and at least $(2-\epsilon)$ times.

  \begin{claim}
  $M(G_n) - M(G_{n+1}) \geq \mu(B_n)$
  \end{claim}
  \begin{proof}[Proof of the claim]
    Partition $B_n$ into finitely many Borel sets $A_{r,s}$ where $x \in A_{r,s}$ iff $x$ has $r$-many $G_n$ neighbors and $s$-many $G_{n+1}$ neighbors (so $r > s$ and $r+s \leq d$).  We compute
    \begin{align*}
      M(G_n)-M(G_{n+1}) &= \int_X |(G_n)_x| - |(G_{n+1})_x|\ d\mu\\
      &\geq \int_{B_n} (2-\epsilon)|(G_n)_x| - (2+\epsilon)|(G_{n+1})_x|\ d\mu\\
      &=\sum_{r,s} \int_{A_{r,s}} (2-\epsilon)r - (2+\epsilon)s\ d\mu\\
      &=\sum_{r,s} \int_{A_{r,s}} 2(r-s) -\epsilon(r+s)\ d\mu\\
      &\geq \sum_{r,s} \int_{A_{r,s}} 2 - d\epsilon\ d\mu\\
      &= \mu(B_n)
    \end{align*}
    as required.
  \end{proof}
  The remainder of the argument is as in the proof of Theorem \ref{thm:invariant}.
\end{proof}

Given Proposition~\ref{prop:quasiinv}, the proof of Theorem~\ref{thm:subexp} is straightforward.

\begin{proof}[Proof of Theorem~\ref{thm:subexp}]
  Fix a degree bound $d$ for $G$ and put $\epsilon = \frac{1}{d}$.  It suffices to construct for each $x \in X$ a $G$-quasi-invariant Borel probability measure $\mu_x$ whose Radon-Nikodym cocycle is $\epsilon$-bounded on $G$ such that $\mu_x(\{x\})>0$.  If we do so, Proposition \ref{prop:quasiinv} ensures that the flip sequence $c_n$ converges $\mu_x$-everywhere for each $x$, and thus it converges everywhere.  The limit is then an unfriendly coloring by the same argument as in the final claim in the proof of Theorem \ref{thm:invariant}.

  To construct $\mu_x$, first define a purely atomic measure $\nu_x$ supported on the $G$-component of $x$ by declaring $\nu_x(\{y\}) = (1+\epsilon)^{-\delta(x,y)}$, where $\delta$ denotes the graph metric.  Subexponential growth of $G$ ensures that $K = \sum_{y \in [x]_G} \nu_x(\{y\}) < \infty$.  Finally, put $\mu_x = \frac{1}{K} \nu_x$.
\end{proof}

\section*{Acknowledgments}
We thank the anonymous referee for insightful comments. We also thank Alekos Kechris for organizing the seminar that inspired this paper.

\bibliography{refs}

\end{document}